\documentclass[a4paper,12pt]{article}
\usepackage{amsmath,amsfonts,amsthm,amssymb,verbatim,graphicx,color}
\usepackage{enumerate, schemata}
\usepackage{comment}
\usepackage{color}

\title{Formulas and Upper Bounds for the Carath{\'e}odory Number of Hamming Graphs}

\usepackage{authblk}

\author[1,2]{Ezequiel Dratman\thanks{\texttt{edratman@campus.ungs.edu.ar}}}
\author[1]{Lucía M. González\thanks{\texttt{lgonzalez@campus.ungs.edu.ar}}}
\author[1,2]{Luciano N. Grippo\thanks{\texttt{lgrippo@campus.ungs.edu.ar}}}

\affil[1]{Universidad Nacional de General Sarmiento. Instituto de Ciencias; Argentina.}
\affil[2]{UNGS, CONICET, ICI, Prov. de Buenos Aires, Argentina.}

\newtheorem{theorem}{Theorem}[section]
\newtheorem{coro}{Corollary}[section]

\newtheorem{lem}{Lemma}[section]
\newtheorem{claim}{Claim}[section]
\newtheorem{prop}{Proposition}[section]
\theoremstyle{definition}
\newtheorem{remark}{Remark}[section]

\begin{document}
\maketitle
%
%
\begin{abstract}
Let $G$ be a simple graph and let $S$ be a subset of its vertices. We say that $S$ is $P_3$-convex if every vertex $v \in V(G)$ that has at least two neighbors in $S$ also belongs to $S$. The $P_3$-hull set of $S$ is the smallest $P_3$-convex set of $G$ that contains $S$. Carath\'eodory number of a graph $G$, denoted by $c(G)$, is the smallest integer $c$ such that for every subset $S \subseteq V(G)$ and every vertex $p$ in the $P_3$-hull of $S$, there exists a subset $F \subseteq S$ with $|F| \leq c$ such that $p$ belongs to the $P_3$-hull of $F$. In this article, we present upper bounds and formulas for the $P_3$-Carathéodory number in Hamming graphs, which are defined as the Cartesian product of $n$ complete graphs.
\end{abstract}
\section{Introduction}\label{intro}
A convexity on a graph can be understood as a rule for spreading ``contamination'' from a given set of initially contaminated vertices. Various such rules have been studied in the literature, leading to the definition of several graph parameters associated with different convexities. In particular, many authors have investigated the behavior of these parameters in well-structured graph classes~\cite{Grippo2021, Bedo2023, Liao2023}, including those defined via graph products~\cite{Mario, Bresar2023}. These problems can be quite challenging, even when the underlying graph has a seemingly simple structure. A seminal result in this direction was provided by Bollob\'as in 2006 in~\cite{bollobas2006art}, motivating a broad line of subsequent research. In 2020, Bo{\v{s}}tjan Bre{\v{s}}ar and Mario Valencia-Pavón studied the $P_3$-hull number in Hamming graphs. The first author, together with other collaborators, also analyzed the $(p, r)$-spreading number on grids, obtaining a closed formula \cite{BDEH-2024}. In a related line of work, Benavidez and coauthors investigated $3$-percolation numbers~\cite{BBLN-2024}, deriving closed formulas for tori, as well as upper and lower bounds for grids and explicit formulas for certain specific cases.

In convex geometry, Carath\'eodory's theorem states that if a point $x \in \mathbb{R}^d$ lies in the convex hull of a set $P$, then $x$ can be expressed as a convex combination of at most $d + 1$ points from $P$ \cite{caratheodory1911}. In other words, there exists a subset $P' \subseteq P$ with at most $d + 1$ points such that $x$ lies in the convex hull of $P'$. Equivalently, $x$ belongs to an $r$-simplex with vertices in $P$, where $r \leq d$. This classical notion has been extended to graphs under various notions of convexity. Over the past five decades, the Carath\'eodory number has been investigated in several graph convexities, yielding numerous results. A selection of these can be found in~\cite{Douchet1988, CM1996, PWW2008, barbarosa, CoelhoDRS2014}.

All graphs considered in this article are finite, undirected, and simple (i.e., they contain no loops or multiple edges). Standard graph-theoretic concepts and definitions not explicitly defined here can be found in~\cite{West01}. Given a graph $G$, we denote its vertex set by $V(G)$ and its edge set by $E(G)$. 

If $X$ is a finite set, we write $|X|$ to denote its cardinality. The \emph{order} of a graph is the number of its vertices, that is, $|V(G)|$. Given $u, v \in V(G)$, we say that $u$ is \emph{adjacent} to $v$ if $uv \in E(G)$. The \emph{neighborhood} of a vertex $u$, denoted by $N_G(u)$, is the set $\{v \in V(G) \colon uv \in E(G)\}$, and the \emph{closed neighborhood} of $u$ is denoted by $N_G[u] = N_G(u) \cup \{u\}$. 

The \emph{degree} of a vertex $u$, denoted by $d_G(u)$, is the number of its neighbors, that is, $|N_G(u)|$. The \emph{length} of a path is the number of edges it contains. The \emph{distance} between two vertices $u$ and $v$ in $G$, denoted by $d_G(u,v)$, is the minimum length among all paths with $u$ and $v$ as endpoints. Given a set of vertices $S \subseteq V(G)$, we denote by $G[S]$ the subgraph of $G$ induced by $S$.
 	
A \emph{convexity} on a graph $G$ is a pair $(V(G), \mathcal{C})$, where $\mathcal{C}$ is a family of subsets of $V(G)$ satisfying the following conditions: $\emptyset \in \mathcal{C}$, $V(G) \in \mathcal{C}$, and $\mathcal{C}$ is closed under intersections; that is, $V_1 \cap V_2 \in \mathcal{C}$ for every $V_1, V_2 \in \mathcal{C}$. Each set in the family $\mathcal{C}$ is called a \emph{$\mathcal{C}$-convex set}.

Let $\mathcal{P}$ be a set of paths in $G$. If $u$ and $v$ are vertices of $G$, the \emph{$\mathcal{P}$-interval of $u$ and $v$}, denoted by $I_{\mathcal{P}}[u,v]$, is the set of all vertices that lie on some path $P \in \mathcal{P}$ with endpoints $u$ and $v$. For a subset $S \subseteq V(G)$, define $I_{\mathcal{P}}[S] = \bigcup_{u,v \in S} I_{\mathcal{P}}[u,v]$. 

Let $\mathcal{C}$ be the family of all vertex subsets $S \subseteq V(G)$ such that, for every path $P \in \mathcal{P}$ with both endpoints in $S$, all vertices of $P$ also belong to $S$. In other words, $\mathcal{C}$ consists of all subsets $S$ such that $I_{\mathcal{P}}[S] = S$. It is easy to verify that $(V(G), \mathcal{C})$ defines a convexity on $G$, and $\mathcal{C}$ is called the \emph{path convexity generated by $\mathcal{P}$}. 

The \emph{$P_3$-convexity} is the path convexity generated by the set of all paths of length two. Equivalently, a \emph{$P_3$-convex set} is a subset $S \subseteq V(G)$ such that every vertex $v \in V(G) \setminus S$ has at most one neighbor in $S$.

The \emph{$\mathcal C$-hull set} of a set $R\subseteq V(G)$ is the minimum $\mathcal C$-convex set of $G$ containing $R$, denoted by $\mathcal{H}_G[R]$. A \emph{$\mathcal C$-hull set} of $G$ is a set of vertices whose $\mathcal C$-hull set is $V(G)$, and the minimum cardinality of a $\mathcal C$-hull set of $G$,  denoted by $h(G)$, is the \emph{$\mathcal C$-hull number}.

The \emph{Carathéodory number} of a graph $G$, denoted by $c(G)$, is the smallest integer $c$ such that for every subset $S \subseteq V(G)$ and every vertex $p \in \mathcal{H}[S]$, there exists a subset $F \subseteq S$ with $|F| \leq c$ satisfying $p \in \mathcal{H}[F]$. The \emph{boundary} of $\mathcal{H}(S)$ is defined as $\partial \mathcal{H}(S) = \mathcal{H}(S) \setminus \bigcup_{u \in S} \mathcal{H}(S \setminus \{u\})$. A set $S \subseteq V(G)$ is called a \emph{Carathéodory set} of $G$ if $\partial \mathcal{H}(S) \neq \emptyset$. Equivalently, the Carathéodory number of $G$ is the largest cardinality of a Carathéodory set of $G$~\cite{CoelhoDRS2014}.

Let $G_i$, for $1 \leq i \leq n$, be arbitrary graphs. The graph $G_1 \square \cdots \square G_n$, known as the \emph{Cartesian product}, is defined as the graph whose vertex set is $V(G_1) \times \cdots \times V(G_n)$. Two vertices $(x_1, \ldots, x_n)$ and $(y_1, \ldots, y_n)$ are adjacent if and only if there exists an index $j \in \{1, \ldots, n\}$ such that $x_j y_j \in E(G_j)$ and $x_i = y_i$ for all $i \neq j$.

If for each $i \in \{1, \ldots, n\}$ the graph $G_i$ is a complete graph $K_{r_i}$, then $H_n(r_1, \ldots, r_n) = K_{r_1} \square K_{r_2} \square \cdots \square K_{r_n}$, with $n > 1$ and $r_i \geq 1$, is called a \emph{Hamming graph}. When the context is clear, we will write $H_n$ instead of $H_n(r_1, \ldots, r_n)$.

The \emph{dimension} of a Hamming graph is the number of its factors, that is, the number of coordinates of its vertices. If $r_i = 1$ for some $1 \leq i \leq n$, we define $H_n(r_1, \ldots, r_n) = K_{r_1} \square \cdots \square K_{r_{i-1}} \square x \square K_{r_{i+1}} \square \cdots \square K_{r_n}$, where $V(K_{r_i}) = \{x\}$ consists of a single vertex. Throughout the paper, we use the set $\{0,\dots, r-1\}$ to refer to the vertex set of a complete $K_r$.

Let $\Sigma$ be a finite alphabet. Given two $r$-tuples $w_1$ and $w_2$ over $\Sigma$, the \emph{Hamming distance} $H(w_1, w_2)$ is the number of positions in which $w_1$ and $w_2$ differ. Vertices of a Hamming graph can be labeled so that the distance between any two vertices coincides with the Hamming distance between their labels.

A graph $G$ is a Hamming graph if and only if each vertex $v \in V(G)$ can be labeled with a tuple $w(v)$ of length equal to the dimension of $G$ such that $H(w(u), w(v)) = d_G(u, v)$ for all $u, v \in V(G)$. Such a labeling is called a \emph{Hamming labeling} (see~\cite{product}).

The article is organized as follows. In Section~\ref{caratheodory number}, we present preliminary concepts and technical results concerning the $P_3$-convexity in graphs, with a particular focus on Hamming graphs. This includes key lemmas on the structure of the $P_3$-hull number and Carathéodory sets. In Section~\ref{sec: upper bounds}, we derive upper bounds for the Carathéodory number in Hamming graphs. We also establish a recursive inequality that relates the Carathéodory number to the dimension of the graph. Section~\ref{sec: caratheodroy number} contains the main results of the paper. We provide closed formulas for the Carathéodory number of Hamming graphs in the case where each factor is a complete graph with at least three vertices. The proof is divided into three cases according to the congruence class of the dimension modulo $3$. Finally, in Section~\ref{sec: open questions}, we summarize our findings and outline some open questions for future research, including conjectures about the Carathéodory number when at least one factor is $K_2$.
\section{Preliminary results}\label{caratheodory number}
In this section, we present preliminary results that will be used throughout the paper. 
\begin{prop}\label{prop barbarosa}\cite[Proposition 2.1 ]{barbarosa}.
Let $G$ be a graph, and let $S$ be a Carath\'eodory set of $G$.
\begin{enumerate}
\item  If $G$ has order at least $2$ and is either complete, or a path, or a cycle,
then $c(G) = 2$.
\item No proper subset $S'$ of $S$ satisfies $\mathcal{H}_G(S') = V (G)$.
\item The convex hull $\mathcal{H}_G(S)$ of $S$ induces a connected subgraph of $G$.
\end{enumerate}
\end{prop}
The following two technical lemmas, due to Bre\v{s}ar and Valencia-Pabon, describe the construction rules for the $P_3$-hull of a set in a Hamming graph.
Let $\{c_{i_1},\,\ldots,\,c_{i_k}\}\subseteq \{1,\,\ldots,\, n\}$, where $1\le i_1<\cdots<i_k\le n$ with $1\le k< n$. A \emph{Hamming subgraph} of a Hamming graph $H_n$ is a subgraph induced by a set of the form $\{x\in V(H_n):\, x_{i_j} = c_{i_j} \mbox{ for all } 1\le j\le k\}$. If $H_n$ is a Hamming graph. We say that such a Hamming subgraph has dimension $k$. It is not hard to prove that $P_3$-hull of a set in a Hamming graph is always a disjoint union of Hamming subgraphs.
\begin{lem}~\cite[Lemma~3]{Mario}\label{lem dist Mario}
	Let $H_n$ be a Hamming graph with $n > 1$ and $r_i > 1$ for $1 \leq i \leq n$. Let $S \subset V(H_n)$ be a set such that $\mathcal{H}[S]$ is a connected subgraph; i.e., $\mathcal{H}[S]$ is a Hamming subgraph with dimension $k$.
\begin{itemize}
	\item If $x \in V(H_n)$ such that $d_{H_n}(x,S)=1$, then $\mathcal{H}(S \cup \{x\})$ induces a Hamming subgraph of dimension $k + 1$.
	\item If $x \in V(H_n)$ such that $d_{H_n}(x,S)=2$, then $\mathcal{H}(S \cup \{x\})$ induces a Hamming subgraph of dimension $k + 2$.
	\item If $x \in V(H_n)$ such that $d_{H_n}(x,S)> 2$, then $\mathcal{H}(S \cup \{x\})$ is the disjoint union of $H[S]$ and the singleton subgraph induced by $\{x\}$..
\end{itemize}
\end{lem}
Given two subsets $S_1$ and $S_2$ of a graph $G$, the \emph{distance} between $S_1$ and $S_2$ is defined as the minimum $d_G(S_1, S_2) = \min\{ d_G(u_1, u_2) \colon u_1 \in S_1,\, u_2 \in S_2 \}$.
\begin{lem}~\cite[Lemma~4]{Mario}\label{lem2 dist Mario}
	Let $H_n$ be a Hamming graph, with $n > 1$ and $r_i > 1$ for $1 \leq i \leq n$,  and let $H^1$ and $H^2$ be two Hamming subgraphs of $H_n$ with dimension $k$ and $k'$, respectively.
\begin{itemize}
	\item If $d_{H_n}(V(H^1), V(H^2)) = 1$, then $\mathcal {H}(V(H^1) \cup V(H^2))$ induces a Hamming subgraph of dimension at most $k + k' + 1$.
	\item If $d_{H_n}(V(H^1), V(H^2)) = 2$, then $\mathcal {H}(V(H^1) \cup V(H^2))$ induces a Hamming subgraph of dimension at most $k + k' + 2$.
	\item If $d_{H_n}(V(H^1), V(H^2)) > 2$, then $\mathcal {H}(V(H^1) \cup V(H^2))$ is the disjoint union of $H^1$ and $H^2$.
\end{itemize}
\end{lem}
%
%

Given an $n$-tuple $x$, we denote by $x^i(\ell)$ the $n$-tuple obtained from $x$ by replacing the $i$-th coordinate $x_i$ with $\ell$. Note that it is possible that $x_i = \ell$.

Let $U$ be a Hamming subgraph of $H_n$ such that, for some fixed $i \in \{1, \ldots, n\}$ and $k \in \{0, \ldots, r_i - 1\}$, every vertex $v \in U$ satisfies $v_i = k$. We define
\[
U^i = \left\{ x^i(\ell) \colon x \in U,\ 1 \leq \ell \leq r_i - 1 \right\}.
\]

In the sequel we write $H_n$ to denote the Hamming graph $H_n(r_1,\dots,\, r_n)$ of dimension $n$, with $n > 1$ and $r_i > 1$ for every $ 1\leq i \leq n$. If $G\cong H_n$ we identify the vertices of $G$ by the corresponding $n$-tuples in $H_n$.

Lemmas~\ref{lem dist Mario} and~\ref{lem2 dist Mario} imply the following remark.
\begin{remark}\label{rmk: replacing-a-coordinate}
	Let $G \cong H_n$, and let $S \subseteq V(G)$ such that $\mathcal{H}(S)$ is connected and $w \in S$.
	
	Suppose that for some $i \in \{1, \ldots, n\}$ and $k \in \{0, \ldots, r_i - 1\}$, all vertices $v \in \mathcal{H}(S)$ satisfy $v_i = k$. Let $\ell \in \{1, \ldots, r_i - 1\}$ with $\ell \neq k$. Then, one of the following conditions holds:
	\begin{enumerate}
		\item If $d_G(w, S \setminus \{w\}) \geq 2$, then
		\[
		\mathcal{H}\big((S \setminus \{w\}) \cup \{w^i(\ell)\}\big) = \mathcal{H}(S \setminus \{w\}) \cup G[\{w^i(\ell)\}].
		\]
		
		\item If $d_G(w, S \setminus \{w\}) \leq 1$, then
		\[
		\mathcal{H}\big((S \setminus \{w\}) \cup \{w^i(\ell)\}\big) = \big(\mathcal{H}(S)\big)^i.
		\]
	\end{enumerate}
\end{remark}
For instance, consider the Hamming graph $G \cong H_6$. Set $S = \{v_1, v_2, v_3\}$, where 
\[v_1 = (1, 0, 0, 0, 0, 0),
v_2 = (0, 1, 0, 0, 0, 0), 
v_3 = (0, 0, 1, 1, 0, 0)
.\]
Note that $\mathcal{H}(S) \cong H_4 \square 0 \square 0$.

Additionally, $\mathcal{H}\big((S \setminus \{v_3\}) \cup \{v_3^5(1)\}\big) = H_2 \square 0 \square 0 \square 0 \square 0 + G[v_3^5(1)]$, where $v_3^5(1) = (0, 0, 1, 1, 1, 0)$. 

If instead $v_3 = (0, 0, 1, 0, 0, 0)$, and thus $\mathcal{H}(S) \cong H_3 \square 0 \square 0 \square 0$, then
\[
\mathcal{H}\big((S \setminus \{v_3\}) \cup \{v_3^4(1)\}\big) = \big(\mathcal{H}(S)\big)^4,
\]
where$
\big(\mathcal{H}(S)\big)^4 \cong H_4 \square 0 \square 0.
$

To preserve a homogeneous notation, we write $H_1$ to denote $K_r$, for some $r \geq 2$. 
\begin{remark} \label{rem:remark 2.1}
	The Carath\'eodory number of Hamming graphs satisfies the following properties:
		\begin{enumerate}
			\item $c(H_i)=2$ for each $i\in\{1,\, 2\}$, 
			\item $c(H_3)=3$, and
			\item $c(H_n)\geq n$, for all $n \geq 4$.
		\end{enumerate}
\end{remark}
\begin{proof}
We know, by  Proposition \ref{prop barbarosa}, that  $c(H_1) = 2$.

Consider the set $S=\{(1,0),\, (0,1)\}$ which is a  Carath\'eodory set of $G\cong H_2$. Indeed, $\partial \mathcal{H}(S)=V(G)\setminus \{(1,0),\,(0,1)\} \neq \emptyset$.  Therefore, $c(G)\geq 2$. 

Towards a contradiction, suppose that $c(H_2)\geq 3$. Let $S$ be a Carath\'eodory set of $H_2$ with $k = \vert S \vert \geq 3$. We claim that $S$ is a hull set of $H_2$. Otherwise, $\mathcal{H}(S)$ induces a subgraph isomorphic to either $K_{r_1} \square b$ or $a \square K_{r_2}$, for some $a \in \{0, \ldots, r_1-1\}$ and $b \in \{0, \ldots, r_2-1\}$. Without loss of generality, assume that $\mathcal{H}(S)$ induces $K_{r_1} \square b$. Hence, $S=\{(i_1,b),\,(i_2,b),\,\ldots,\, (i_k,b)\}$ with $0\le i_1\le\cdots\le i_k$ is a Carath\'eodory set of $K_{r_1} \square b\cong K_{r_1}$, where $k = c(H_2)$. This contradicts that  $c(K_{r_1}) = 2$

Notice that $S=\{(0,1,1),\, (1,0,1),\, (1,1,0)\}$ is a Carath\'eodory set of $G\cong H_3$. Indeed, 
\[\partial \mathcal{H}(S) =  V(G)\setminus (V( 1 \square K_{r_2} \square K_{r_3})\cup V( K_{r_1} \square 1 \square K_{r_3})\cup V( K_{r_1} \square K_{r_2} \square 1)).\] 
In particular, $(0,0,0) \in \partial \mathcal{H}(S)$, so $c(G)\geq 3$. 

Suppose, towards a contradiction, that $c(G)\geq 4$. Let $S =  \{x_1,\,x_2,\,\ldots,\,x_k\}$ be a Carath\'eodory set of $H_3$ with $k =\vert S \vert \geq 4$. By symmetry, it suffices to consider two cases for $G[\mathcal{H}(S \setminus \{x_k\})]$: $K_{r_1} \square b \square c\cong K_{r_1}$ or $K_{r_1} \square K_{r_2} \square  c$, for some $b \in \{0, \ldots, r_2-1\}$ and $c \in \{0, \ldots, r_3-1\}$. In either case, it is easy to see that $k-1\le 2$, so $k\le 3$, a contradiction. Therefore, $c(G) = 3$.

Let be $G \cong H_n$ and consider the set $S=\{v_1, v_2, \cdots, v_n\}$ with $n\ge 4$, where each $v_i$ is defined  by $(v_i)_j=0$ if $j=i$ and $(v_i)_j=1$ if $j \neq i$. It is clear that $S$ is a hull set; i.e., $\mathcal{H}(S)=V(G)$. 
By Lemma~\ref{lem dist Mario}, for each $1 \leq i \leq n$ we have:
\[G\big[\mathcal{H}(S \setminus \{v_i\})\big]= K_{r_1} \square K_{r_2} \square  \cdots \square K_{r_{i-1}} \square 1 \square K_{r_{i+1}} \square  \cdots \square K_{r_n}.\]

Therefore, $(0,0, \cdots, 0) \in \partial H(S)$ and $S$ is a Carath\'eodo\-ry set of $G$. Thus, $c(G)\geq n$.
\end{proof}
The following result provides a lower bound on the dimension of the hull set of a maximun Carath\'eodory set.
\begin{lem}
Let $G \cong H_n$. If $S$ is a maximun Carath\'eodory set of $G$, then $dim(G[\mathcal{H}(S)])\geq n-1$ for all $n \geq 4$.
\end{lem}
\begin{proof}
Let $S$ be a maximun Carath\'eodory set of $G$. By Proposition~\ref{prop barbarosa}, the subgraph $G[\mathcal{H}(S)]$ is connected.

Suppose, towards a contradiction, that $dim(G[\mathcal{H}(S)]) \leq n-2$. Without loss of generality, we can assume that $G[\mathcal{H}(S)] = H_h\square 0 \square \cdots \square 0$, with $h \leq n-2$. Since $S$ is a Carath\'eodory set, there exists a vertex 
\[u=(d_1, \cdots,d_h, 0,\cdots,0) \in \partial \mathcal{H}(S),\]
where $0\leq d_i \leq r_i-1$ for every $1 \leq i \leq h$. 

Let $S' = S\cup \{ w \}$, where  $w_i = d_i$  for every $1\le i\le h$, $w_{h + 1} = w_{h + 2} = 1$, and $w_j = 0$ for every $h + 3\le j\le n$. By lemma \ref{lem dist Mario}, we have 
\[G[\mathcal{H}(S')]=H_{h + 2}\square 0 \square \cdots \square 0.\] 
Define $\widehat{w}$ as follows: $\widehat{w}_1 =d_1+1$ (sum should be considered modulo $r_1$), $\widehat{w}_i = w_i$ for every $2\le i\le n$. 

On the one hand, it is clear that $\widehat{w} \in \mathcal{H}(S')$ and $\widehat{w} \notin \mathcal{H}(S'\setminus \{w\})=\mathcal{H}(S)$. On the other hand, it is easy to verify that $d(w, \mathcal{H}(S\setminus\{v\}))\geq 3$ for every $v \in S$. Therefore, by Lemma \ref{lem dist Mario}, 
\[\mathcal{H}(S'\setminus \{v\})=\mathcal{H}(S \cup \{w\} \setminus \{v\})=\mathcal{H}(S\setminus \{v\}) + G[\{w\}],\]
for every $v \in S$. 

Since $d(\widehat{w}, \mathcal{H}(S\setminus \{v\})) >2$ for every $v \in S$ and $\widehat{w} \neq w$, it follows that $\widehat{w} \notin \mathcal{H}(S'\setminus \{v\})$ for every $v\in S$. We have already proved that $\widehat{w} \in \partial \mathcal{H}(S')$. Therefore, $S'$ is a Carath\'eodory set with $\vert S'\vert > \vert S \vert$, contradicting the assumption that $S$ is a maximum Carath\'eodory set. The contradiction arises from assuming that $dim(G[\mathcal{H}(S)]) \leq n-2$. Therefore, $dim(G[\mathcal{H}(S)])\ge n-2$.
\end{proof}
The following two technical results lay the theoretical groundwork for the next section.
\begin{lem}\label{lem:k+1}
Let $G \cong H_n$. If $S$ is a Carath\'eodory set of $G$ such that $dim(G[\mathcal{H}(S)])=k < n$, then there exists a Carath\'eodory set $T$ such that $|T|\ge |S|$ and $dim(G[\mathcal{H}(T)])=k+1$.
\end{lem}
\begin{proof}
	Let $S=\{v_1,\ldots,v_r\}$ be a Carath\'eodory set of $G$. Without loss of generality, we may assume that $dim(G[\mathcal{H}(S)])= n-1$ and that $(v_i)_n = 0$ for every $1\le i\le r	$. Since $S$ is a Carath\'eodory set of $G$,  by Lemma \ref{lem dist Mario}, it follows that $0 < d(v_i, \mathcal{H}(S\setminus\{v_i\}))\leq 2$ for every $1 \leq i \leq r$. 

	First, assume that there exists $i$, $1 \leq i \leq r$, such that $d(v_i, \mathcal{H}(S\setminus\{v_i\}))=1$. Without loss of generality, suppose that $d(v_1, \mathcal{H}(S\setminus\{v_1\}))=1$. Let $v_1=(d_1,\ldots,d_{n-1},0)$ and define $v_{r+1}=v_1^n(1)$. 

	Suppose that $H = G[\mathcal{H}(S \setminus \{v_1\})] = C_1 + \cdots + C_k$, where the $C_i$'s are the connected components of $H$, and thus they are Hamming subgraphs of $G$. 

	Since $V(H)$ is a convex set of $G$, there exists exactly one index $1 \leq i\leq k$ such that $d(v_1, C_i)=1$. Moreover, $v_1$ is adjacent to only one vertex $w \in V(C_i)$, and  $d(v_1, C_j)>1$ for all $j \neq i$. 

	Let $ T = (S\setminus\{v_1\})\cup \{v_{r+1}\}$. Since $d(v_1,C_i)=1$ and $d(v_1,v_{r+1})=1$, it follows that $v_1 \in \mathcal{H}(T)$. Therefore, $\mathcal{H}(T) = V(G)$. 

	Let $v=(y_1,\cdots,y_{n-1},0) \in \partial \mathcal{H}(S)$, and define $v'= v^n(1)$. Clearly, $v' \notin \mathcal{H}(T \setminus \{v_{r+1}\})$. 

	Notice that for every $1 < j <r + 1$, $T\setminus\{v_j\} = S\setminus\{v_1,v_j\}\cup\{v_{r+1}\}$. Suppose that $\mathcal{H}(S\setminus \{v_j\}) = C_{1,j} + \cdots + C_{k_j, j}$ for each $1<j <r + 1$. By Remark~\ref{rmk: replacing-a-coordinate} and Lemma~\ref{lem2 dist Mario}, for every such $j$, $\mathcal{H}(T\setminus \{v_j\}) = C_{1,j}^n + \cdots + C_{k_j, j}$ and thus $v'\notin \mathcal{H}(T \setminus \{v_{r+1}\})$. 

	We have already shown that $v'\in\partial \mathcal{H}(T)$. Therefore, $T$ is a Carath\'eodory set of $G$ with $\textrm{dim}(\mathcal{H}(T)) = \textrm{dim}(\mathcal{H}(S)) + 1$. 

	Finally, assume that $d(v_j, G[\mathcal{H}(S\setminus\{v_j\}))=2$ for every  $1 \leq j \leq r$. Without loss of generality, suppose that $v_1=(d_1,\ldots,d_{n-1},0)$, and define $v_{r+1}=v_1^n(1)$. Let $T = S\cup\{v_{r + 1}\}$. By Lemma~\ref{lem dist Mario}, we have $\mathcal{H}(T)=V(G)$. 

	Suppose that $H_j = G[\mathcal{H}(S \setminus \{v_j\}]) = C_{1, j} + \cdots + C_{k_j,j}$, where the $C_{i,j}$'s are the connected components of $H_j$, and thus they are Hamming subgraphs of $G$. 

	Let $v=(y_1,\cdots,y_{n-1},0) \in \partial \mathcal{H}(S)$, and define $v' = v^n(1)$. Clearly, $v'\notin \mathcal{H}(T\setminus\{v_{r + 1}\})$. 

	On the one hand, $d(v_{r + 1}, C_{h,1})\ge 3$ for every $1\le h\le k_1$, by Lemma~\ref{lem dist Mario} we have $\mathcal{H}(T\setminus\{v_1\}) = \mathcal{H}(S\setminus \{v_1\})\cup\{v_{r + 1}\}$. Since $v\neq v_1$ and $v\notin \mathcal{H}(S\setminus \{v_1\})$, it follows that $v'\neq v_{r + 1}$ and $v'\notin \mathcal{H}(T\setminus\{v_1\})$. 

	On the other hand, assume, without loss of generality, that $v_1\in V(C_{1, j})$ for every $2\le j\le r$. By Remark~\ref{rmk: replacing-a-coordinate} and Lemma~\ref{lem2 dist Mario}, for every $2\le j\le r$, $\mathcal{H}(T\setminus\{v_j\}) =  C_{1, j}^n + \cdots + C_{k_j,j}$ and therefore $v'\notin \mathcal{H}(T \setminus \{v_j\})$.  We have already shown that $v'\in\partial \mathcal{H}(T)$. Therefore, $T$ is a Carath\'eodory set of $G$ with $\textrm{dim}(\mathcal{H}(T)) = \textrm{dim}(\mathcal{H}(S)) + 1$.
\end{proof}

We conclude the section with the following result.

\begin{coro}\label{cor: caratheodory set with hull of maximum dimension} 
	Let $G \cong H_n$. Then there exists a Carath\'eodory set $S$ such that $\dim(G[\mathcal{H}(S)]) = n$. Additionally, $S$ is a hull set of $G$. 
\end{coro}
\section{Upper bounds}\label{sec: upper bounds}

Given a graph $G$, a \emph{minimal hull set} is a hull set $S$ such that $\mathcal{H}(S)=V(G)$ and $\mathcal{H}(R)\subsetneq V(G)$ for all $R \subset S$. We denote by $p(G)$ the maximum cardinality of a minimal hull set. 

The following remark compares the parameter $p(H_n)$  in terms of the dimension of the Hamming graph.  
\begin{remark}\label{rem:remark 2..2}
 	$p(H_{n}) \geq p(H_{n-1})$.
\end{remark}
The next remark establishes a relation between $c(H_n)$ and $p(H_n)$. 
\begin{remark} \label{rem:remark 2.2}
	Let $G \cong H_n$. Then the following statements hold:
		\begin{itemize}
			\item If $S$ is a maximum Carath\'eodory set of $G$ such that $\textrm{dim}(\mathcal{H}(G[S])) = n$, then $S$ is a minimal hull set of $G$, and
			\item  $ c(H_n)\le p(H_n)$.
		\end{itemize}
\end{remark}
\begin{proof}
	Let $S$ be a maximun Carath\'eodory set of $G\cong H_n$ such that $G[\mathcal{H}(S)]$ has maximum dimension. By Corollary~\ref{cor: caratheodory set with hull of maximum dimension}, $S$ is a minimal hull set of $G$. 

	Lemma~\ref{lem:k+1} guaranties the existence of a maximum Crath\'eodory set of dimension $n$. Therefore, $c(H_n)\le p(H_n)$.
\end{proof}
The following technical lemma plays a central role in obtaining an upper bound for the Carath\'eodory number of a Hamming graph.
\begin{lem}\label{lemma: p(h_n)}
	Let $G \cong H_n$. Then, $p(H_{n})\le\max\{2p(H_{n-3}),p(H_{n-1}) + 1\}$ for every $n\ge 4$.
\end{lem}
\begin{proof}
	Let $S$ be a minimal hull set of $G$ with the maximum number of vertices, and let $T \subsetneq S$ be a subset of maximum cardinality such that $G[\mathcal{H}(T)]$ is connected and $dim(\mathcal{H}(T))< n$. 	
	Assume that $dim(\mathcal{H}(T))=n - 1$. By Lemma~\ref{lem dist Mario}, since $T$ is a minimum hull set, it follows that $S\setminus T = \{v\}$. Therefore, $p(H_n(r_1,\dots,\, r_n))\le p(H_{n - 1})$.
	Now assume that $\textrm{dim}(\mathcal{H}(T)) = n - 2$.  Let $v\in S\setminus T$. By Lemma~\ref{lem dist Mario}, we have $d(v,\mathcal{H}(T)) = 2$. Otherwise, $G[\mathcal{H}(T\cup\{v\})]$ would be a Hamming subgraph of dimension $n - 1$, contradicting the maximality of $T$. Thus, $S = T\cup\{v\}$. 
	Therefore, by Remark~\ref{rem:remark 2..2}, $p(H_n)\le p(H_{n-2}) + 1\le p(H_{n - 1}) + 1$.
	Finally, assume that $\textrm{dim}(\mathcal{H}(T)) \leq n-3$. By Remark~\ref{rem:remark 2..2}, we have $\vert T \vert \leq p(H_{n-3})$.
	
	Suppose, towards a contradiction, that $G[\mathcal{H}(S \setminus T)]$ is disconnected, and $G[\mathcal{H}(S \setminus T)]= C_1 + \cdots + C_k$ where the $C_i$'s are Hamming subgraphs of $H_n$. By Lemma~\ref{lem dist Mario},  there exists $i \in \{1,\,\ldots,\,k\}$ such that $d(C_i, \mathcal{H}(T))\leq 2$. 
	
	Thus, $G[\mathcal{H}(T \cup C_i)]$ is connected, and $\textrm{dim}(G[\mathcal{H}(T)]) < \textrm{dim}(G[\mathcal{H}(T \cup C_i)])< n$, contradicting the maximality of $T$. Hence, $G[\mathcal{H}(S \setminus T)]$ is connected. 
	
	Since $T$ is of maximum cardinality under the given conditions, we have  $\vert S \vert = \vert S \setminus T \vert + \vert T \vert \leq p(H_{n-3})+ p(H_{n-3})$. Therefore, $p(H_n(r_1,\dots,\, r_n)) \leq p(H_{n-3})+ p(H_{n-3})=2p(H_{n-3})$.
\end{proof}
\begin{coro}\label{cor: caratheodory number for small haming graphs}
	Let $H_n$  with $ 1\leq n\le 6$. Then, $c(H_n) = n$.
\end{coro}
\begin{proof}
	By Remarks \ref{rem:remark 2.1} and~\ref{rem:remark 2.2}, $n\leq c(H_n) \leq p(H_n)$.
	By Remark \ref{rem:remark 2.1} and a direct case analysis, it can ve proved that $p(H_3) = c(H_3) = 3$. 
	
	Also recall that $p(H_1) = p(H_2) = 2 = c(H_1) = c(H_2)$. 
	
	Therefore, the result follows from Lemma~\ref{lemma: p(h_n)}.
\end{proof}
The following lemma provides an upper bound for the Carath\'eodory number of Hamming graphs of dimension greater than or equal to $7$. These bounds will used in the next section to prove the main result of this article.
\begin{lem} \label{lem: qn} 
	Let $H_n$ be the Hamming graph of dimension $n \geq 7$. Then, the following upper bounds hold:
		\begin{enumerate}
			\item If $n\equiv 0 \;(\bmod \;{3})$, then $p(H_n) \leq 3\cdot2^{\frac{n}{3}-1}$.
			\item If $n\equiv 1 \;(\bmod \;{3})$, then $p(H_n) \leq 4 \cdot 2^{\frac{n-1}{3}-1}$.
			\item If $n\equiv 2 \;(\bmod \;{3})$, then $p(H_n) \leq 5\cdot2^{\frac{n-2}{3}-1}$.
		\end{enumerate}
\end{lem} 
\begin{proof}
	Let $\{q_k\}_{k\in\mathcal N}$ be the sequence defined as follows: set $q_1 = 2$,  $q_k = k$ for every $2\le k\le 6$, and for every $n\geq 7$ define $q_n=\max\{2q_{n-3},q_{n-1}+1\}$ . 
	
	By applying Lemma~\ref{lemma: p(h_n)} and using induction, we obtain $p(H_n) \le q_n$, for every $2\le n$.

	By direct computation, $q_7 = 8$, $q_8 =10$, and $q_9 = 12$. It is easy to prove by induction that, for every $n\geq 7$, $q_n=\max\{2q_{n-3},q_{n-1}+1\}=2q_{n-3}$ . Thus, $p(H_n)\le 2 q_{n - 3}$ for every $n\ge 7$. 
	
	Therefore, the results follows by induction on $n$.
\end{proof}
\section{Carath\'eodory number of a Hamming graphs obtained from complete graphs with at least three vertices}\label{sec: caratheodroy number}
In this section $H_n$ represents a Hamming graph with dimension at least $7$ and each complete of the product has at least three vertices.

In the previous section, we computed the Carth\'eodory number for Hamming graphs with dimension at most six (see Crollary~\ref{cor: caratheodory number for small haming graphs}). 

Now we are ready to prove our main result.
\begin{theorem}\label{thm: main result}
	Let $H_n$ be a Hamming graph of dimension $n$ with $n \geq 7$ and each complete of the product has at least three vertices.
	\begin{enumerate}
		\item If $n\equiv 0 \;(\bmod \;{3})$, then $c(H_n) = 3\cdot2^{\frac{n}{3}-1}$.
		\item If $n\equiv 1 \;(\bmod \;{3})$, then $c(H_n) = 4\cdot2^{\frac{n-1}{3}-1}$.
		\item If $n\equiv 2 \;(\bmod \;{3})$, then $c(H_n) = 5\cdot2^{\frac{n-2}{3}-1}$.
	\end{enumerate}
\end{theorem}
\begin{proof}
	We will only provide the details for the case $n\equiv 0 \;(\bmod \;{3})$. The other two cases follow analogously, so we will leave their details to the readers and only sketch their proofs. 
	
	By Lemma \ref{lem: qn}, we have an upper bound for $c(H_n)$. Our goal is to construct a  Carath\'eodory in $H_n$ whose cardinality matches this upper bound every $n\geq 7$. 

	\vspace{0,5cm}
	\textbf{Case $n\equiv 0 \;(\bmod \;{3})$} 
\vspace{0,5cm}
	
	We recursively define the set of $3\cdot2^{\frac n 3 - 1}$ vertices in $H_{n - 2}$ as follows. Let $n = 3q$ with $q\ge 3$, Consider the following vertices in $H_7$: $w_1^{(9)}=(0,0,0,0,0,0,0)$, $w_2^{(9)}=(0,1,2,2,0,0,0)$, $w_3^{(9)}=(0,1,1,1,0,0,0)$, $w_4^{(9)}=(1,2,2,2,2,1,0)$, $w_5^{(9)}=(1,2,0,0,1,1,0)$, $w_6^{(9)}=(1,2,1,1,1,1,0)$, $v_1^{(9)}=(2,2,2,2,2,2,2)$, $v_2^{(9)}=(2,2,2,0,0,1,2)$, $v_3^{(9)}=(2,2,2,1,1,1,2)$, $v_4^{(9)}=(2,1,0,0,0,0,1)$, $v_5^{(9)}=(2,1,1,2,2,0,1)$, and $v_6^{(9)}=(2,1,1,1,1,0,1)$.

	For every $1 \leq i \leq 3\cdot 2^{q-3}$ and for every $n\ge 12$ such that $n\equiv 0 \;(\bmod \;{3})$ and $n = 3q$ with $q \geq 3$, define: $w_i^{(n)}=(0,w_i^{(n-3)},0,0)$, $w_{3\cdot 2^{q-3} + i}^{(n)}=(1,v_i^{(n-3)},1,0)$, $v_i^{(n)}=(2,w_i^{(n-3)},0,2)$, and $v_{3\cdot 2^{q-3} + i }^{(n)}=(2,2,v_i^{(n-3)},1)$.

	Finally, define the following sets:  $S_w^{(n)}=\lbrace w_i^{(n)}: 1\leq i \leq 3\cdot 2^{q-2} \rbrace$ and $S_v^{(n)}=\lbrace v_i^{(n)}: 1\leq i \leq 3\cdot 2^{q-2}\rbrace$.
	
	We claim the following:
	\begin{claim}\label{claim: main result}
		For every $n \geq 9$ such that $n \equiv 0 \pmod{3}$ and $n = 3q$ with $q \geq 3$, the following statements hold:
			\begin{enumerate}
			\item $S_w^{(n)}$ is a minimum hull set of $H_{n-3}\square 0$. Moreover, for every $1\leq j\leq 3\cdot 2^{q-2}$, if $x\in \mathcal{H}(S_{w}^{(n)}\setminus \lbrace w_j^{(n)}\rbrace)$, then $x_1\neq 2$ .
			\item $S_v^{(n)}$ is a minimal hull set of $2\square H_{n-3}$. Moreover, for every $1\leq j\leq 3\cdot 2^{q-2}$, if $x\in \mathcal{H}(S_v^{(n)}\setminus \lbrace v_j^{(n)}\rbrace)$, then $x_{n-2} \neq 0$. 
		\end{enumerate}
	\end{claim}
Using case analysis together with Lemmas~\ref{lem dist Mario} and~\ref{lem2 dist Mario}, it is straightforward to verify the claim holds for $n = 9$. 

Assume by inductive hypothesis that Claim~\ref{claim: main result} holds for every $n = 3k$ with $3\le k< q$. Let $n = 3q$ with $q\ge 4$. 
 
On the one hand, by inductive hypothesis, we have $\mathcal{H}\left(S_{w}^{(n)}\right)=\mathcal{H}(S_1\cup S_2)$, where $\mathcal{H}(S_1)\cong 0\square H_{n - 6}\square 0\square 0\square 0$, $\mathcal{H}(S_2)\cong 1\square 2\square H_{n - 6}\square 1\square 0$. Moreover, the following properties hold: 
\begin{itemize}
	\item For every $1 \leq i \leq 3\cdot 2^{q-3}$, if $x\in \mathcal{H}(S_1\setminus\{w_i^{(n)}\})$ then $x_2\neq 2$. In addition,  $d(\mathcal{H}(S_1\setminus\{w_i^{(n)}\}),\mathcal{H}(S_2))\ge 3$.
	\item For every $3\cdot 2^{q-3} + 1 \leq i \leq 3\cdot 2^{q-2}$, if $y\in \mathcal{H}(S_2)\setminus\{w_i^{(n)}\})$, then $y_{n-4}\neq 0$. In addition, $d(\mathcal{H}(S_2\setminus\{v_i^{(n)}\}),\mathcal{H}(S_1))\ge 3$, .
\end{itemize}
Hence, by Lemma~\ref{lem2 dist Mario}, $S_w^{(n)}$ is a minimum hull set of $H_{n-3}\square 0$ such that, for every $1\le i\le 3\cdot 2^{q-2}$, $x\in (S_w^{(n)}\setminus\{w_i^{(n)}\})$  implies $x_1\neq 2$.

On the other hand, by inductive hypothesis, we have $\mathcal{H}\left(S_{v}^{(n)}\right)=\mathcal{H}(T_1\cup T_2)$, where $\mathcal{H}(T_1)\cong 2\square H_{n - 6}\square 0\square 0\square 2$, $\mathcal{H}(T_2)\cong 2\square 2\square 2\square H_{n - 6}\square  1$. Moreover, the following properties hold: 
\begin{itemize}
	\item For every $1 \leq i \leq 3\cdot 2^{q-3}$, if $x\in \mathcal{H}(T_1\setminus\{v_i^{(n)}\})$, then $x_2\neq 2$. In addition, $d(\mathcal{H}(T_1\setminus\{v_i^{(n)}\}),\mathcal{H}(T_2))\ge 3$.
	\item For every $3\cdot 2^{q-3} + 1 \leq i \leq 3\cdot 2^{q-2}$, $y\in \mathcal{H}(T_2\setminus\{v_i^{(n)}\})$, then $y_{n-3}\neq 0$. In addition,  $d(\mathcal{H}(T_2\setminus\{v_i^{(n)}\}),\mathcal{H}(T_1))\ge 3$.
\end{itemize}
Hence, by Lemma~\ref{lem2 dist Mario}, $S_v^{(n)}$ is a minimum hull set of $2\square H_{n-3}$ such that, for every $1\le i\le 3\cdot 2^{q-2}$, $x\in (S_v^{(n)}\setminus\{v_i^{(n)}\})$  implies $x_{n - 2}\neq 0$. Therefore, Claim~\ref{claim: main result} holds.

Let $U_1^{(n)}=\{(u_1,0,0):\, u_1\in S_w^{(n)}\}$ and  $U_2^{(n)}=\{(u_2,1,1):\, u_2\in S_v^{(n)}\}$. 

Furthermore, by Claim~\ref{claim: main result} and Lemma~\ref{lem2 dist Mario}, the set $U = U_1^{(n)}\cup U_2^{(n)}$ is a hull set of $H_n$. By Claim~\ref{claim: main result}, the following properties hold: $d(\mathcal{H}(U_1^{(n)}\setminus\{u\}),\mathcal{H}(U_2^{(n)}))\ge 3$ for every $u\in U_1^{(n)}$, and $d(\mathcal{H}(U_2^{(n)}\setminus\{u\}),\mathcal{H}(U_1^{(n)}))\ge 3$ for every $u\in U_2^{(n)}$. Moreover, by construction, for every $u\in U$, $x\notin\mathcal{H}(U\setminus\{u\})$  where $x_i = 2$ for each $i\in\{n-1,n\}$. 

Therefore, $U$ is a Charath\'eodory set of $H_n$ with $|U| = 3\cdot 2^{q - 1}$. By Lemma~\ref{lemma: p(h_n)}, it follows that $\textrm{c}(H_n)= 3\cdot 2^{\frac n 3 - 1}$.

\vspace{0,5cm}
\textbf{Case $n\equiv 1 \;(\bmod \;{3})$}
\vspace{0,5cm}

We recursively define the set of $4\cdot 2^{\frac {n-1} 3 - 1}$ vertices in $H_{n - 2}$ as follows. Let $n = 3 q + 1$ with $q\ge 2$. Consider the following vertices in $H_5$:\\
$w_1^{(7)}=(0,0,0,0,0)$, $w_2^{(7)}=(0,1,0,0,0)$, $w_3^{(7)}=(1,2,1,1,0)$,\\
$w_4^{(7)}=(1,2,2,1,0)$, $v_1^{(7)}=(2,2,2,2,2)$, $v_2^{(7)}=(2,2,2,1,2)$,\\
$v_3^{(7)}=(2,1,1,0,1)$, and $v_4^{(7)}=(2,1,0,0,1)$.

For every $1 \leq i \leq 4\cdot 2^{q-3}$ and for every $n\ge 10$ such that $n\equiv 1 \;(\bmod \;{3})$ and $n = 3q + 1$ with $q\ge 3$, $w_i^{(n)}=(0,w_i^{(n-3)},0,0)$, $w_{4\cdot 2^{q-3}+i}^{(n)}=(1,v_i^{(n-3)},1,0)$, $v_i^{(n)}=(2,v_i^{(n-3)},2,2)$, and $v_{4\cdot 2^{q-3}+i}^{(n)}=(2,w_i^{(n-3)},1,1)$.

Finally, define the following sets:  $S_w^{(n)}=\lbrace w_i^{(n)}: 1\leq i \le 4\cdot 2^{q-2} \rbrace$ and $S_v^{(n)}=\lbrace v_i^{(n)}: 1\leq i \leq  4\cdot 2^{q-2}\rbrace$.

Let $U_1^{(n)}=\{(u_1,0,0):\, u_1\in S_w^{(n)}\}$ and $U_2^{(n)}=\{(u_2,1,1):\, u_2\in S_v^{(n)}\}$. It can be proved, similarly to the case $n\equiv 0 \;(\bmod \;{3})$, that $U = U_1^{(n)}\cup U_2^{(n)}$ is a Carth\'eodory set with $|U| = 4\cdot 2^{q-1}$.  

\vspace{0,5cm}
\textbf{Case $n\equiv 2 \;(\bmod \;{3})$}
\vspace{0,5cm}

We recursively define the set of $5\cdot2^{\frac {n-2} 3 - 1}$ vertices in $H_{n - 2}$ as follows. Let $n = 3 q + 2$ with $q\ge 2$. Consider the following vertices in $H_6$:\\ 
$w_1^{(8)} = (1,2,2,2,2,0)$, $w_2^{(8)}=(0,0,0,0,0,0)$, $w_3^{(8)} = (0,0,1,0,0,0)$,\\ 
$w_4^{(8)} = (0,1,2,1,1,0)$, $w_5^{(8)} = (0,1,2,2,1,0)$,\\
$v_1^{(8)} = (2,2, 2, 2, 2, 1)$, $v_2^{(8)} = (2, 2, 2, 1, 2, 1)$, $v_3^{(8)} = (2, 1, 0, 0, 1, 1)$,\\
$v_4^{(8)} = (2, 1, 0,0, 1, 1)$, and $v_5^{(8)} = (2, 0, 0, 0, 0 , 2)$.

For every $1 \leq i \leq 5 \cdot 2^{q-3}$ and for every $n\ge 11$ such that $n\equiv 2 \;(\bmod \;{3})$ and $n = 3q + 2$, $w_i^{(n)} = (0,w_i^{(n-3)},0,0)$, $w_{5\cdot 2^{q-3} + i}^{(n)}= (1,v_i^{(n-3)},1,0)$, $v_i^{(n)} = (2,w_i^{(n-3)},1,1)$, and $v_{5 \cdot 2^{q-3} + i }^{(n)}=(2,v_i^{(n-3)},2,2)$.

Finally, define the following sets:  $S_w^{(n)}=\lbrace w_i^{(n)}: 1\leq i \leq 5\cdot 2^{q-2} \rbrace$ and $S_v^{(n)}=\lbrace v_i^{(n)}: 1\leq i \leq 5\cdot 2^{q-2}\rbrace$.

Let $U_1^{(n)}=\{(u_1,0,0):\, u_1\in S_w^{(n)}\}$ and $U_2^{(n)}=\{(u_2,1,1):\, u_2\in S_v^{(n)}\}$. It can be proved, similarly to the case $n\equiv 0 \;(\bmod \;{3})$, that $U = U_1^{(n)}\cup U_2^{(n)}$ is a Carath\'edory set with $|U| = 5\cdot 2^{q-1}$. 
\end{proof}
\section{Conclusions and open questions}\label{sec: open questions}
In this article, we provide a closed formula for the Carathéodory number of Hamming graphs defined as the Cartesian product of complete graphs with at least three vertices (see Theorem~\ref{thm: main result}). To the best of our knowledge, the only existing closed formula for Hamming graphs related to $P_3$-convexity is the hull number formula given by Bres{\v a}r and Valencia-Pabon in~\cite{Mario}. 

Additionally, an upper bound for the Carathéodory number of any Hamming graph can be obtained by combining Remark~\ref{rem:remark 2.2} with Lemma~\ref{lem: qn}. 

We conjecture that for every $n \geq 7$, if the Hamming graph $H_n$ includes at least one factor $K_2$ in its Cartesian product, then $c(H_n) < p(H_n)$. Improving the general upper bound in this case and providing a closed formula remain open problems.
%


\end{document}